\documentclass[10pt,twoside]{siamart1116}

\usepackage[english]{babel}
\usepackage{graphicx,epstopdf,epsfig, verbatim}
\usepackage{amsfonts,epsfig,fancyhdr,graphics,amsmath,amssymb}

\newcommand{\HE}{Name of Handling Editor}
\newcommand{\DoS}{Month/Day/Year}
\newcommand{\DoA}{Month/Day/Year}
\newcommand{\CA}{Andrada Pojar}

\newcommand{\Names}{Andrada Pojar}
\newcommand{\Title}{A bound for the nilpotence index associated to \lowercase{$m$}-nil-clean nonderogatory matrices}

\newtheorem{remark}[theorem]{Remark}
\newtheorem{example}[theorem]{Example}

\newcommand{\F}{\mathbb{F}}
\newcommand{\M}{\mathcal{M}}

\newcommand{\Z}{{\mathbb Z}}

\newcommand{\Tr}{{\mathrm{Tr}}}

\renewcommand{\theequation}{\thesection.\arabic{equation}}
\renewtheorem{theorem}{Theorem}[section]

\begin{document}

\bibliographystyle{plain}

\setcounter{page}{1}

\thispagestyle{empty}

 \title{\Title\thanks{Received
 by the editors on \DoS.
 Accepted for publication on \DoA.
 Handling Editor: \HE. Corresponding Author: \CA}}

\author{
Andrada Pojar\thanks{Department of Mathematics,
Technical University of Cluj-Napoca, 28 Memorandumului Street, 400114, Cluj-Napoca, Romania.}
}

\markboth{\Names}{\Title}

\maketitle

\begin{abstract}
 It is proved that if $\F$ is a field of positive characteristic $p,$ and if $m$ and $n$ are positive integers such that $m\geq2,$ and $4 \leq n\leq p\leq mn-1,$ for every $n\times n$ nonderogatory matrix $A\in \M_n(\F),$ with trace in $\{k.1_{\F}\mid k\in \{0,1,\dots,p-1\}\},$ there exist $m$ idempotent matrices $E_1, E_2,\dots, E_m,$ and a nilpotent matrix $N$, such that $A=E_1+E_2+\dots+E_m+N,$ with $N^k=0,$ where $k=n$ if $p\in \{nm-1,nm-2\},$ $k=n-1$ if  $p=nm-3,$ otherwise $k=\mathrm{max}(2,1+\lfloor\frac{n-1}{r}\rfloor),$ if $n$ is even, and $k=\mathrm{max}(3,1+\lfloor\frac{n-1}{r}\rfloor),$ if $n$ is odd,  where $r:=\lfloor\frac{nm-p}{2}\rfloor.$ Moreover, for $4\leq n>p,$ $A$ is the sum of two idempotent matrices, and a square zero one, if $n$ is even,  and it is sum of two idempotent matrices and one which third power is zero, if $n$ is odd.
\end{abstract}

\begin{keywords}
$m$-nil-clean; nonderogatory matrix; companion matrix; nilpotent; idempotent.
\end{keywords}
\begin{AMS}
15A24, 15A83, 16U99.
\end{AMS}

\section{Introduction}

Let $\F$ be a field with positive characteristic $p.$
As usual, $\mathcal{M}_n(\F)$ will denote the full matrix ring of $n\times n$ matrices over $\F.$

A square matrix is {\it nonderogatory}, or also {\it cyclic} if its characteristic
and minimal polynomials coincide. A matrix is nonderogatory if and only if it is similar to a
companion matrix $C$.
We recall that a companion matrix $C\in \mathcal{M}_n(\F)$ is a matrix of the form:
$$C=C_{c_0,c_1,\ldots, c_{n-1}}=\left(\begin{array}{ccccc}
0 & 0  &\ldots & 0 & -c_0 \\
1 & 0  &\ldots & 0 & -c_1\\
\vdots & \vdots  &\cdots  & \vdots & \vdots \\
0 & 0 &\ldots & 1 & -c_{n-1}
 \end{array}\right).$$

 A square matrix $A$ over $\F$ is nil-clean if there exist an idempotent matrix $E,$ and a nilpotent matrix $N,$ such that $A=E+N.$ Let $m\geq 2$ be an integer. We consider decompositions of nonderogatory matrices $A,$ such that there exist idempotents $E_i$, $i\in \{1,2,\dots,m\},$ and a nilpotent $N,$ with $A=E_1+E_2\dots+E_m+N,$ which are $m$-nil-clean decompositions.

Nil-clean decompositions are related to clean ones, introduced by Nicholson in \cite{N}, when investigating exchange rings, and were first studied by Diesl in \cite{Diesl}. An important result appeared in \cite{BCDM} and \cite{KLZ} about them is: every $n\times n$ matrix over a division ring $D$ is nil-clean if and only if $D=\F_2.$

A theme related to a set $S$ of nilpotents of a ring is the study of boundness of the nilpotence index of S -- finding a positive integer $n$, as small as possible, such that $a^n=0$ for all $a\in S$ , and it was first studied in \cite{KWZ}. \v{S}ter has proved that there exists a nil-clean decomposition of every matrix of $\mathbb{M}_n(\F_2),$ with nilpotent $N$ having nilpotence index at most $4$ (i.e $N^4=0.$). Moreover this result has sharpness, that is, there exist $4\times 4$ companion matrices over $\F_2,$ which cannot be decomposed as a sum of an idempotent and a nilpotent matrix of nilpotence index at most $3.$

An extension of nil-clean decompositions to finite fields of odd cardinality $q,$ was done in \cite{AM}: every matrix over such a field is the sum of a $q$-potent $E=E^q,$ and a nilpotent. In \cite{B} there is even more -- the nilpotent $N$ involved in such a decomposition can be with nilpotence index at most $3$ (i.e. $N^3=0).$

In \cite{BrMe} it has been proved that every nil-clean nonderogatory $n\times n$ matrix $A,$ over a field of positive characteristic $p,$ with trace of $A$ not equal to $1,$ has a decomposition $A=E+V,$ such that $E^2=E,$ and $V^{p+1}=0.$ If $\mbox{trace}(A)=1,$ then there is a similar decomposition with $V^{p+2}=0.$

Another topic related to our purpose are matrices over a positive characteristic field, that are sums of idempotents. Their investigation is in \cite{Seguins}.
We study decompositions of nonderogatory matrices that are sums of $m$ idempotents, and a nilpotent. By Lemma 3.1 from \cite{C} we know that the trace of such a matrix is an integer multiple of unity of $\F.$
By the assumption $mn-1\geq p,$ every nonderogatory matrix with trace being an integer multiple of unity of $\F$ is $m$-nil-clean, considering the fact that companion matrices with trace in $\{1.1_{\F},2.1_{\F},\dots,(mn-1).1_{\F} \}$ are $m$-nil-clean, by \cite{C}. We prove that if $\F$ is a field of positive characteristic $p,$ and if $m$ and $n$ are positive integers such that $m\geq2,$ and $4\leq n\leq p\leq mn-1,$ for every $n\times n$ nonderogatory matrix $A\in \M_n(\F),$ with trace in $\{k.1_{\F}\mid k\in \{0,1,\dots,p-1\}\},$ there exist $m$ idempotent matrices $E_1, E_2,\dots, E_m,$ and a nilpotent matrix $N$, such that $A=E_1+E_2+\dots+E_m+N,$ with $N^k=0,$ where $k=n$ if $p\in \{nm-1,nm-2\},$ $k=n-1$ if  $p=nm-3,$ otherwise $k=\mathrm{max}(2,1+\lfloor\frac{n-1}{r}\rfloor),$ if $n$ is even, and $k=\mathrm{max}(3,1+\lfloor\frac{n-1}{r}\rfloor),$ if $n$ is odd,  where $r:=\lfloor\frac{nm-p}{2}\rfloor.$ Moreover, for $4\leq n>p,$ $A$ is the sum of two idempotent matrices, and a square zero one, if $n$ is even, and it is sum of two idempotent matrices and one which third power is zero, if $n$ is odd.
\section{$m$-nil-clean decompositions up to some diagonal matrix}

\begin{definition}
We say that a matrix $C$, is $m$-nil-clean up to a diagonal $D_0$ if there exist $m$ idempotent matrices $E_1,E_2,\dots, E_m,$  a nilpotent matrix $N,$ and a diagonal matrix $D_0,$ such that $C=E_1+E_2+\dots+E_m+N-D_0.$
\end{definition}

We point out a lemma, which will be called "the fitting lemma", which is a restatement of Lemma $1$ from \cite{Pojar-comp}, and it is an essential tool in our approach.

\begin{lemma}(The fitting lemma)\mbox{ }
Let $n$ be a positive integer, $\F$ a field, $C\in \M_n(\F),$ a companion matrix, and $D'\in \M_n(\F),$ a diagonal matrix with same trace as $C.$ Then there exists a companion matrix $C',$ with trace $0,$ such that $C\sim C'+D'.$
\end{lemma}

We provide a way to obtain $m$-nil clean decompositions, for an $n\times n$ companion matrix, up to some diagonal matrix, with nilindex of the involved nilpotent less than $n,$  by considering a good partition of $n-1.$

\begin{definition}
We call a nontrivial partition of $n-1=d_1+d_2+\dots+d_r,$ into positive integers, to be a {\it good partition} if $d_i>1,$ for all $i\in [[2,r]].$
\end{definition}

\begin{definition}
We say an $n\times n$ nilpotent $N$ is associated to the good partition $n-1=d_1+\dots+d_r,$ if there exist matrices $G_1,\dots,G_r$ such that $G_1$ is arbitrary, the first row of each of the matrices $G_2,\dots, G_r$ is a zero row, and $N$ is of the following form:
$$N=\left(\begin{array}{cccc}
J_{d_1} & {[}0{]} & {[}0{]} & G_1 \\
{[}0{]} & \ddots & \mbox{ } & \vdots \\
\vdots & \mbox{ } & J_{d_r} & G_r \\
{[}0{]} & \mbox{ } & {[}0{]} & 0
\end{array}\right),$$
where $J_k$ denotes the lower triangular nilpotent Jordan cell of size $k.$
\end{definition}

\begin{remark}\label{max}
An $n\times n$ nilpotent $N$ associated to the partition $n-1=d_1+d_2+\dots+d_r,$ has nilindex no more than $\mathrm{max}(d_1+1,d_2,\dots,d_r)$ (this uses the critical observation that $G_2,\dots, G_r$ have their first row zero).
\end{remark}

\begin{definition}
We say an $n\times n$ nilpotent $N$ is associated to the trivial partition  $n-1=d_1,$ if there exists the matrix $G_1$ such that $G_1$ has the first row zero, and $N$ is of the following form:
$$N=\left(\begin{array}{cccc}
J_{d_1} & G_1 \\

{[}0{]} & 0
\end{array}\right),$$
where $J_{d_1}$ denotes the lower triangular nilpotent Jordan cell of size $d_1.$
\end{definition}

\begin{remark}\label{n-1bound}
An $n\times n$ nilpotent $N$ associated to the trivial partition $n-1=d_1,$ has nilindex no more than $n-1$ (this uses the critical observation that $G_1$ has its first row zero). We emphasize that if $G_1$ was arbitrary, then we would have obtained the bound $n$ for the nilindex of $N.$
\end{remark}

\begin{definition}\label{naiv}
We say that the companion matrix $C$ is $m$-nil-clean with respect to the good or trivial partition $n-1=d_1+d_2+\dots+d_r,$ if $C\sim E_1+E_2+\dots+E_m+N,$ where $E_1,E_2,\dots,E_m$ are idempotents, and $N$ is a nilpotent associated to the good or trivial partition.
\end{definition}

\begin{definition}
We call a decomposition for a companion matrix $C= E_1+E_2+\dots+E_m+N-D_0$ to be an $m$-nil-clean decomposition up to the diagonal $D_0,$ with respect to the good or trivial partition $n-1=d_1+d_2+\dots+d_r,$ if $E_1, E_2,\dots, E_m$ are idempotents, and $N$ is a nilpotent associated to the good or trivial partition.
\end{definition}

\begin{example}\label{std}
Let $C$ be an $n\times n$ companion matrix, over a field $\F$ of positive characteristic $p,$ with $\Tr(C)\in \{k.1_{\F}\mid k\in \{0,1,\dots,p-1\} \}$, $n>2,$ a good partition $n-1=d_1+d_2+\dots+d_r,$ and $m\geq 2$ an integer such that $nm-1\geq p\geq n.$ Consider the following matrices:
\\$F_1=\sum_{i=2}^r E_{d_1+\dots+d_{i-1}+1,d_1+\dots+d_{i-1}+1},$ $F_2$ the matrix that has the same rows as $C$ for the indices in $\{d_1+1,d_1+d_2+1,\dots,d_1+\dots+d_{i-1}+1\}$, and zero rows elsewhere, $F_3=$ the sum of any number of matrices of the form $E_{i,i}$, with either
$i<d_1,$ or $i\in [[d_1+\dots+d_{j-1}+2,d_1+\dots+d_j-1]],$ for some $j\in [[2,r-1]],$ or $i\in [[d_1+\dots+d_{r-1}+2,d_1+\dots+d_r]].$ Because the partition is good, $E_1=F_1+F_2+F_3$ is idempotent. Moreover, the trace of $E_1$ can take any value of the form $k.1_F,$ with $k\in [[r-1,n-r]].$
Now take $E_2$ an idempotent of the form $\sum_{i=1}^kE_{i,i}+(E_{n,n-1}+E_{n,n}),$ with $k\leq n-2.$ We observe that the possible traces for $E_2$ are elements of the form $\{k.1_{\F}\mid 1\leq k\leq n-1\}.$ We'll be calling $E_1,$ and $E_2$ "severance matrices". Let $E_3,\dots, E_m$ be diagonal idempotents. Each of them can have any possible trace in
$\{k.1_{\F}\mid 0\leq k\leq n\}.$ We have that there exist a nilpotent $N$ associated to the partition $n-1=d_1+d_2+\dots+d_r,$ and a diagonal $D_0,$ (which depends only on the choice of the diagonal entries of
$E_i$'s) such that $C=E_1+E_2+\dots+E_m+N-D_0.$

We point out for further use that the possible traces for the matrix $\sum_{k=1}^mE_k$ constitute a set of the form $\{k.1_{\F},\mid k\in [[[a,b]]\},$
for some integers $a$ and $b,$ such that $b-a=(n-2r+1)+(n-2)+(m-2)n=mn-2r-1,$ with the $n-2r+1$ term coming from the slack with respect to the choice of the diagonal entries of $E_1,$ and so on.
\end{example}

\begin{example}\label{d1}
Let $C$ be an $n\times n$ companion matrix, over a field $\F$ of positive characteristic $p,$  with $\Tr(C)\in \{k.1_{\F}\mid k\in \{0,1,\dots,p-1\} \}$,  a trivial partition $n-1=d_1,$ and $m\geq 2$ an integer such that $nm-1\geq p\geq n.$ Consider the following matrices:
$F_1$- the matrix with the same first row as $C,$ and zero rows elsewhere, $F_2=\sum_{k=1}^{j}E_{k,k},$ for some $j\in [[2,n-1]].$ Let the severance matrix $E_1$ be the idempotent $F_1+F_2.$ The trace of $E_1$ can take any value of the form $k.1_{\F},$ with $k\in [[1,n-1]].$ We consider $E_2,\dots,E_m$ as in Example \ref{std}. We have that there exist a nilpotent $N$ associated to the partition $n-1=d_1,$ and a diagonal $D_0$ (which depends only on the choice of the diagonal entries of
$E_i$'s), such that $C=E_1+E_2+\dots+E_m+N-D_0.$

We point out for further use that the possible traces for the matrix $\sum_{k=1}^mE_k$ constitute a set of the form $\{k.1_{\F},\mid k\in [[[a,b]]\},$
for some integers $a,$ and $b,$ such that $b-a=(n-2)+(n-2)+(m-2)n=nm-4.$
\end{example}

\begin{example}\label{diagE1}
Let $C$ be an $n\times n$ companion matrix, over a field $\F$ of positive characteristic $p,$  with $\Tr(C)\in \{k.1_{\F}\mid k\in \{0,1,\dots,p-1\} \},$ and $m\geq 2$ an integer such that $nm-1\geq p\geq n.$
If we change in Example \ref{std} the severance matrix $E_1,$ with a diagonal idempotent matrix, we get an example of $m$-nil-clean decomposition for $C$, up to some diagonal matrix. We note that the involved nilpotent is not associated to a partition of $n-1,$ having instead the first $n-1$ rows of $C,$ and a zero last row.

We point out for further use that the possible traces for the matrix $\sum_{k=1}^mE_k$ constitute a set of the form $\{k.1_{\F},\mid k\in [[[a,b]]\},$
for some integers $a,$ and $b,$ such that $b-a=n+(n-2)+n(m-2)=nm-2.$
\end{example}

\section{$m$-nil-clean nonderogatory matrices}

We denote by $\textrm{diag}(A)=\textrm{diag}(B)$ that matrices $A$ and $B$ have the same diagonal entries.

\begin{lemma}\label{zero}
Let $C$ be an $n\times n$ companion matrix $C,$ over a field $\F$ of positive characteristic $p,$ $m\geq 2$ an integer, $n-1=d_1+d_2+\dots+d_r,$ a good or trivial partition. Assume that in an $m$-nil-clean decomposition for $C$, up to a diagonal $D_0$, with respect to the taken partition, idempotents $E_1,E_2,\dots E_m$ are as in one of the Examples \ref{std}, \ref{d1}, \ref{diagE1}, such that $\Tr(C-E_1-E_2-\dots-E_m)=0.$ Then $C$ is $m$-nil-clean.
\end{lemma}
\begin{proof}
Let $D'$ be a diagonal matrix such that $\textrm{diag}(D')=\textrm{diag}(E_1+\dots+E_m).$ So $\Tr(D')=\Tr(C).$ By The fitting lemma, there exists a companion $C',$ with trace $0$ such that $C\sim C'+D'.$
 Let $E_1'$ as $E_1$ in Example \ref{std}, or Example \ref{d1} respectively, but obtained from lines of $C',$ and $E_1'=E_1$ in case of Example \ref{diagE1}.
 Then there exist nilpotent matrix $N'$ and diagonal $D_0',$ such that $C'=E_1'+E_2+\dots+E_m+N'-D_0'$ is an $m$-nil-clean decomposition for $C',$ up to the diagonal $D_0'$ (with respect to the given good or trivial partition in the first two cases).
 We have $\textrm{diag}(E_1)=\textrm{diag}(E_1'),$ and $E_1'$ differs from $E_1$ only in its last column, and not in its last entry. We get $D_0'=D_0,$ and
 $\textrm{diag}(E_1'+E_2+\dots+E_m)=\textrm{diag}(E_1+E_2+\dots+E_m)=\textrm{diag}(D').$
Since $\textrm{diag}(C')=\textrm{diag}(0),$ and $\textrm{diag}(N')=\textrm{diag}(0)$, it follows that $\textrm{diag}(E_1'+E_2+\dots+E_m)=\textrm{diag}(D_0).$ Therefore $\textrm{diag}(D')=\textrm{diag}(D_0)$, and since $D',$ and $D_0$ are diagonals, we get that $D'=D_0.$ In conclusion $C\sim E_1'+E_2+\dots+E_m+N'.$
\end{proof}

\begin{remark}\label{key}
 We can obtain an $m$-nil-clean decomposition for a companion matrix $C$, over a field of positive characteristic $p,$ with $\Tr(C)\in \{k.1_{\F}\mid k\in \{0,1,\dots,p-1\} \},$ provided that $b-a\geq p-1,$ where $a$ and $b$ are integers such that the possible traces for the sum of all idempotents from a decomposition like in one of the Examples \ref{std}, \ref{d1}, \ref{diagE1}, constitute a set of the form $\{k.1_{\F},\mid k\in [[a,b]]\}.$
Indeed, in this case all possible traces in $ \{k.1_{\F}\mid k\in \{0,1,\dots,p-1\} \}$ can be reached, and hence there exists a decomposition as in one of the three examples, with
idempotents $E_1,E_2,\dots,E_m,$ such that $\Tr(E_1+\dots+E_m)=\Tr(C).$ Then by Lemma \ref{zero} we get an $m$-nil-clean decomposition for $C.$

Moreover, if we have a good partition $n-1=d_1+\dots+d_r$ (case of Example \ref{std}),  the involved nilpotent of the $m$-nil-clean decomposition is similar to a nilpotent associated to that good partition, so we have a nilindex no larger than $\mathrm{max}(d_1+1,d_2,\dots,d_r),$ by Remark \ref{max}. For the case of a trivial partition $n-1=d_1$ (case of Example \ref{d1}), the involved nilpotent of the $m$-nil-clean decomposition is similar to a nilpotent associated to that trivial partition, so we have a nilindex no larger than $n-1,$ by Remark \ref{n-1bound}.
\end{remark}

In the following theorem we are about to prove that some nonderogatory matrices are $m$-nil-clean. Since for any nonderogatory matrix $A$ there exists a companion matrix $C,$ such that $A\sim C,$ it is sufficient to prove the statement for companion matrices.

\begin{theorem}
Let $\F$ be a field of positive characteristic $p,$ and $n\geq 4$ and $m\geq 2$ be positive integers, such that $nm-1\geq p\geq n.$ Then for every $n\times n$ nonderogatory matrix $A\in \M_n(\F),$ with trace in $\{k.1_{\F}\mid k\in \{0,1,\dots,p-1\}\},$ there exist $m$ idempotent matrices $E_1, E_2,\dots, E_m,$ and a nilpotent matrix $N$, such that $A=E_1+E_2+\dots+E_m+N,$ with $N^k=0,$ where $k=n$ if $p\in \{nm-1,nm-2\},$ $k=n-1$ if  or $p=nm-3,$ otherwise $k=\mathrm{max}(2,1+\lfloor\frac{n-1}{r}\rfloor),$ if $n$ is even, and $k=\mathrm{max}(3,1+\lfloor\frac{n-1}{r}\rfloor),$ if $n$ is odd, where $r:=\lfloor\frac{nm-p}{2}\rfloor.$
\end{theorem}
\begin{proof}

Let $C$ be the companion matrix such that $A\sim C.$ We can assume that $n\geq 2,$ since in the case $n=1,$ the only idempotents are $0$ and $1,$ and the only nilpotent is $0,$ so the conclusion here is proved.
\begin{itemize}
\item Case $p\notin \{nm-1,nm-2,nm-3\},$ and $n\geq 4.$
\begin{itemize}
\item Case $r<\frac{n+1}{2}.$ Take integers $q,$ and $a\in [[0,r-1]],$ obtained by computing the Euclidean division $n-1=rq+a,$ and the partition $n-1=\underbrace{q+\dots+q }_{r-a \textrm{ times }q}+\underbrace{(q+1)\dots+(q+1) }_{a \textrm{ times }q+1}.$ We have $rq+a=n-1>2r-2.$
    We obtain the following subcases:
    \begin{itemize}
    \item $q=0.$  We obtain $r<\frac{n+1}{2}< n-1=a\leq r-1,$ which is a contradiction.
    \item $q=1.$ We have $r<\frac{n+1}{2}=\frac{r\cdot 1+a+2}{2}.$ It follows that $a\geq r-1\geq a.$ Therefore $q=1,$ and $a=r-1,$ which leads to a good partition.
    \item $q>1,$ then the partition is good.
    \end{itemize}
    Since the partition is good, there exists an $m$-nil-clean decomposition for $C,$ up to some diagonal matrix, with respect to this partition, like in Example \ref{std}. There, the slack for the trace of the sum of idempotents was $b-a=nm-2r-1\geq p-1$ (because $r=\lfloor\frac{nm-p}{2}\rfloor$). Hence, by Remark \ref{key}, we get that there exist $E_1,\dots,E_m$ idempotent matrices, and $N$ a nilpotent, such that $C=E_1+\dots+E_m+N,$ with $N^k=0,$ where $k=q+1=1+\lfloor \frac{n-1}{r}\rfloor.$

\item Case $r\geq \frac{n+1}{2},$ then $nm-p\geq n+1.$ It follows that  $(n-2)+(m-2)n\geq p-1.$ Take the good partition $n-1=2+\dots+2,$ if $n$ is odd, or $n-1=1+2+\dots+2,$ if $n$ is even.
    Since the partition is good, there exists an $m$-nil-clean decomposition for $C,$ up to some diagonal matrix, with respect to this partition, like in Example \ref{std}. There, the slack for the trace of the sum of the last $m-1$ idempotents was $(n-2)+(m-2)n\geq p-1.$  Hence, by Remark \ref{key}, we get that there exist $E_1,\dots,E_m$ idempotent matrices, and $N$ a nilpotent, such that $C=E_1+\dots+E_m+N,$ with $N^2=0,$ if $n$ is even, and $N^3=0,$ if $n$ is odd. This is due to the fact that $\textrm{max}(1+1,2,\dots,2)=2,$ while $\textrm{max}(2+1,2,\dots,2)=3.$
\end{itemize}
In conclusion, for this principal case there exist $E_1,\dots,E_m$ idempotent matrices, and $N$ a nilpotent, such that $C=E_1+\dots+E_m+N,$ with $N^k=0,$ where $k=\textrm{max}(2,1+\lfloor \frac{n-1}{r}\rfloor),$ if $n$ is even, and $k=\textrm{max}(3,1+\lfloor \frac{n-1}{r}\rfloor),$ if $n$ is odd.
\item Case $p=nm-3.$ Here we have $r=\lfloor\frac{nm-p}{2} \rfloor=1.$ Take the trivial partition $n-1=d_1.$ Then, there exists an $m$-nil-clean decomposition for $C,$ up to some diagonal matrix, with respect to this partition, like in Example \ref{d1}. There, the slack for the trace of the sum of idempotents was $b-a=nm-4\geq p-1,$ and by Remark \ref{key}, we get that there exist $E_1,\dots,E_m$ idempotent matrices, and $N$ a nilpotent, such that $C=E_1+\dots+E_m+N,$ with $N^k=0,$ where $k=n-1.$
\item Case $p\in \{nm-2,nm-1 \}.$ We are not claiming a better nilindex bound than n. There exists an $m$-nil-clean decomposition for $C,$ up to some diagonal matrix, like in Example \ref{diagE1}. There, the slack for the trace of the sum of idempotents was $b-a=nm-2\geq p-1.$ Hence, by Remark \ref{key}, we get that there exist $E_1,\dots,E_m$ idempotent matrices, and $N$ a nilpotent matrix, such that $C=E_1+\dots+E_m+N.$
\end{itemize}
\end{proof}

\begin{remark}
Assume $4\leq n>p.$ Then with $m=2,$ we have $n(m-2)+(n-1)\geq p>p-1;$ hence this case can then be treated as we did for the case $r\geq \frac{n+1}{2},$ of the theorem.  This yields that any nonderogatory matrix $A\in \M_n(\F),$ with trace in $\{k.1_{\F}\mid k\in \{0,1,\dots,p-1\}\},$ is the sum of two idempotent matrices, and a square zero one, if $n$ is even, and it is sum of two idempotent matrices and one with nilindex no larger than $3,$ if $n$ is odd.

\end{remark}
\begin{remark}
 Assume $n= 3$, $p\in \{2,3\}.$ Take $m=2.$ Take the trivial partition $3-1=2.$  Then, there exists an $m$-nil-clean decomposition for $C,$ up to some diagonal matrix, with respect to this partition, like in Example \ref{d1}.  There, the slack for the trace of the sum of idempotents was $b-a=nm-4=2\geq p-1,$ and by Remark \ref{key}, we get that there exist $E_1,E_2$ idempotent matrices, and $N$ a nilpotent, such that $C=E_1+E_2+N,$ with $N^k=0,$ where $k=n-1=2.$
 \end{remark}

\textbf{Acknowledgements.}\mbox{ } The author thanks the reviewer for the suggestions, comments, and remarks, which led to an improvement of this paper.

\end{document}